\documentclass[a4paper,11pt]{article}

\usepackage{graphics}
\usepackage{amsmath,amsthm}

\makeatletter
\thm@notefont{\fontseries\mddefault\upshape}
\thm@headfont{\scshape}
\makeatother

\theoremstyle{plain}
\newtheorem{theorem}{Theorem}
\newtheorem*{theorem*}{Theorem}
\newtheorem{proposition}[theorem]{Proposition}
\newtheorem*{proposition*}{Proposition}
\newtheorem{lemma}[theorem]{Lemma}
\newtheorem*{lemma*}{Lemma}
\newtheorem{corollary}[theorem]{Corollary}
\newtheorem*{corollary*}{Corollary}
\newtheorem*{remark*}{Remark}
\newtheorem*{acknowledgment}{Acknowledgment}

\def\SQ{\mathbf{Q}}

\def\SC{\mathbf{C}}
\def\SP{\mathbf{P}}

\def\supp{\mathop{\rm Supp}\nolimits}
\def\KB{\mathop{\bar{\kappa}}\nolimits}

\def\NV{\mathop r\nolimits}
\def\TA{\mathbin{\ast}}
\def\TP#1{{\vphantom{#1}}^{\mathit{t}}{#1}}
\def\NS{\mathop{o}\nolimits}
\def\TW#1{{t}_{#1}}
\def\PG{\overline{P}}

\title{On Orevkov's rational cuspidal plane curves}
\author{Keita Tono}
\date{}

\begin{document}
\maketitle
\def\thefootnote{}
\def\thefootnote{1}
\begin{abstract}
In this note,
we consider rational cuspidal plane curves having exactly one cusp
whose complements have logarithmic Kodaira dimension two.
We classify such curves with the property that
the strict transforms of them
via the minimal embedded resolution of the cusp have
the maximal self-intersection number.
We show that the curves given by the classification
coincide with those constructed by Orevkov.
\end{abstract}
\section{Introduction}
Let $C$ be a curve on $\SP^2=\SP^2(\SC)$.
A singular point of $C$ is said to be a \emph{cusp}
if it is a locally irreducible singular point.
We say that $C$ is \emph{cuspidal} (resp.~\emph{unicuspidal})
if $C$ has only cusps (resp.~one cusp) as its singular points.
We denote by $\KB=\KB(\SP^2\setminus C)$
the logarithmic Kodaira dimension
of the complement $\SP^2\setminus C$.
Let $C'$ denote the strict transform of a rational unicuspidal plane curve $C$
via the minimal embedded resolution of the cusp of $C$.
By \cite{y},
$\KB=-\infty$ if and only if $(C')^2>-2$.
By \cite[Proposition 2]{ts},
there exist no rational cuspidal plane curves with $\KB=0$.
See also \cite{ko1,or}.
Thus $\KB\ge 1$ if and only if $(C')^2\le -2$.
In \cite{to:uck1},
rational unicuspidal plane curves with $\KB=1$
have already been classified.
It was Orevkov \cite{or}
who constructed two sequences
$C_{4k}$, $C_{4k}^{\ast}$ ($k=1,2,\ldots$)
of rational unicuspidal plane curves with $\KB=2$.
See Section~\ref{sec:or} for details.
The purpose of this note is to classify
rational unicuspidal plane curves $C$ with $\KB=2$ and $(C')^2=-2$.
The main result of this note is the following:
\begin{theorem}\label{thm1}
Let $C$ be a rational unicuspidal plane curve with $\KB=2$.
Then
$C$ is projectively equivalent to one of the Orevkov's curves
if and only if $(C')^2=-2$.
\end{theorem}

For a plane curve $C$,
we denote by $\PG_m(\SP^2\setminus C)$
the logarithmic $m$-genus
of the complement $\SP^2\setminus C$.
In \cite{ko2},
the curve $C_4$ was characterized by $\KB$ and $\PG_4$.
The following theorem characterizes
$C_4$ and $C_4^{\ast}$
by $\KB$,  $\PG_2$ and $\PG_3$.
\begin{theorem}\label{thm:pm}
A reduced plane curve $C$ is projectively equivalent to $C_4$ or $C_4^{\ast}$
if and only if
$\KB(\SP^2\setminus C)\ge 0$ and
$\PG_2(\SP^2\setminus C)=\PG_3(\SP^2\setminus C)=0$.
\end{theorem}
\section{Preliminaries}
In this section, we prepare for the proof of our theorems.
%
\subsection{Linear chains}
Let $D$ be a divisor on a smooth surface $V$,
$\varphi:V'\rightarrow V$ a composite of
successive blowings-up
and $B\subset V'$ a divisor.
We say that
$\varphi$ \emph{contracts} $B$ to $D$,
or simply that $B$ \emph{shrinks to} $D$
if
$\varphi(\supp{B})=\supp{D}$ and
each center of blowings-up of $\varphi$
is on $D$ or one of its preimages.
Let $D_1,\ldots,D_r$ be the irreducible components of $D$.
We call $D$ an \emph{SNC-divisor} if
$D$ is a reduced effective divisor,
each $D_i$ is smooth,
$D_iD_j\le 1$ for distinct $D_i,D_j$,
and $D_i\cap D_j\cap D_k=\emptyset$ for distinct $D_i,D_j,D_k$.
Assume that $D$ is an SNC-divisor
and that each $D_i$ is projective.
Let $\Gamma=\Gamma(D)$ denote the dual graph of $D$.
We give the vertex corresponding to a component $D_i$
the weight $D_i^2$.
We sometimes do not distinguish between $D$
and its weighted dual graph $\Gamma$.
We use the following notation and terminology
(cf.~\cite[Section 3]{fu} and \cite[Chapter 1]{mits1}).
A blowing-up at a point $P\in D$
is said to be \emph{sprouting} (resp.~\emph{subdivisional})
\emph{with respect to} $D$
if $P$ is a smooth point (resp.~node) of $D$.
A component $D_i$ is called a \emph{branching component} of $D$
if $D_i(D-D_i)\ge 3$.

Assume that $\Gamma$ is connected and linear.
In cases where $r>1$,
the weighted linear graph $\Gamma$ together with
a direction from an endpoint to the other
is called a \emph{linear chain}.
By definition,
the empty graph $\emptyset$
and a weighted graph consisting of a single vertex without edges
are linear chains.
If necessary, renumber $D_1,\ldots,D_r$ 
so that the direction of the linear chain $\Gamma$ is from $D_1$ to $D_r$
and $D_iD_{i+1}=1$ for $i=1,\ldots,r-1$.
We denote $\Gamma$ by $[-D_1^2,\ldots,-D_r^2]$.
We sometimes write $\Gamma$ as $[D_1,\ldots,D_r]$.
The linear chain is called \emph{rational} if every $D_i$ is rational.
In this note, we always assume that every linear chain is rational.
The linear chain $\Gamma$ is called \emph{admissible} 
if it is not empty and $D_i^2\le -2$ for each $i$.
Set $\NV(\Gamma)=r$.
We define
the \emph{discriminant} $d(\Gamma)$ of $\Gamma$
as the determinant of the $r\times r$ matrix $(-D_i D_j)$.
We set $d(\emptyset)=1$.

Let $A=[a_1,\ldots,a_r]$ be a linear chain.
We use the following notation if $A\ne\emptyset$:
\[
\TP{A}:=[a_r,\ldots,a_1],\ 
\overline{A}:=[a_2,\ldots,a_r],\ 
\underline{A}:=[a_1,\ldots,a_{r-1}].
\]
The discriminant $d(A)$ has the following properties (\cite[Lemma 3.6]{fu}).
\begin{lemma}\label{lem:det1}
Let $A=[a_1,\ldots,a_r]$ be a linear chain.
\begin{enumerate}
\item[\textnormal{(i)}]
If $r>1$, then
$d(A)=a_1 d(\overline{A})-d(\overline{\overline{A}})=d(\TP{A})=a_r d(\underline{A})-d(\underline{\underline{A}})$.
\item[\textnormal{(ii)}]
If $r>1$, then
$d(\overline{A})d(\underline{A})-d(A)d(\underline{\overline{A}})=1$.
\item[\textnormal{(iii)}]
If $A$ is admissible,
then $\gcd(d(A),d(\overline{A}))=1$ and $d(A)>d(\overline{A})>0$.
\end{enumerate}
\end{lemma}
%

Let $A=[a_1,\ldots,a_r]$ be an admissible linear chain.
The rational number $e(A):=d(\overline{A})/d(A)$
is called the \emph{inductance} of $A$.
By \cite[Corollary 3.8]{fu}, the function
$e$ defines a one-to-one correspondence between the
set of all the admissible linear chains and the set of rational numbers
in the interval $(0,1)$.
For a given admissible linear chain $A$,
the admissible linear chain $A^{\ast}:=e^{-1}(1-e(\TP{A}))$ is called
the \emph{adjoint} of $A$ (\cite[3.9]{fu}).
Admissible linear chains and their adjoints have the following properties
(\cite[Corollary 3.7, Proposition 4.7]{fu}).
\begin{lemma}\label{lem:indf}
Let $A$ and $B$ be admissible linear chains.
\begin{enumerate}
\item[\textnormal{(i)}]
If $e(A)+e(B)=1$, then $d(A)=d(B)$ and $e(\TP{A})+e(\TP{B})=1$.
\item[\textnormal{(ii)}]
We have $A^{\ast\ast}=A$, $\TP{(A^{\ast})}=(\TP{A})^{\ast}$ and
$d(A)=d(A^{\ast})=d(\overline{A^{\ast}})+d(\underline{A})$.
\item[\textnormal{(iii)}]
The linear chain $[A,1,B]$ shrinks to $[0]$
if and only if $A=B^{\ast}$.
\end{enumerate}
\end{lemma}

For integers $a$, $n$ with $n\ge 0$, we define 
$[(a)_n]=[\overbrace{a,\ldots,a}^n]$, $\TW{n}=[2_n]$.
For non-empty linear chains $A=[a_1,\ldots,a_r]$, $B=[b_1,\ldots,b_s]$,
we write
$A\TA B=[\underline{A},a_r+b_1-1,\overline{B}]$,
$A^{\ast n}=\overbrace{A\TA\cdots\TA A}^n$, where $n\ge 1$.
We remark that $(A\TA B)\TA C=A\TA(B\TA C)$
for non-empty linear chains $A$, $B$ and $C$.
By using Lemma~\ref{lem:det1} and Lemma~\ref{lem:indf},
we can show the following lemma.
\begin{lemma}\label{lem:adj}
Let $A=[a_1,\ldots,a_r]$ be an admissible linear chain.
\begin{enumerate}
\item[(i)]
For a positive integer $n$, we have $[A,n+1]^{\ast}=\TW{n}\TA A^{\ast}$.
\item[(ii)]
We have $A^{\ast}=\TW{a_r-1}\TA\cdots\TA\TW{a_1-1}$.
\item[(iii)]
If there exist positive integers $m$, $n$ such that
$[A,m+1]=[n+1,A]$
(resp.~$A\TA\TW{m}=\TW{n}\TA A$),
then $m=n$,
$a_1=\cdots=a_r=n+1$
(resp.~$A=\TW{n}^{\ast\NV(A^{\ast})}$).
\end{enumerate}
\end{lemma}

The following two lemmas
describe the processes of contractions of special linear chains.
The first one can be proved easily.
We prove the second one.
\begin{lemma}\label{lem:bu1}
Let $A$ be an admissible linear chain and $B$ a non-empty linear chain.
Suppose that a composite $\pi$ of blowings-down contracts $[A,1]$ to $B$.
\begin{enumerate}
\item[(i)]
The linear chain $B$ is the image of the first $\NV(B)$ curves of $A$.
We have $A=B\TA\TW{n}$, where $n=\NV(A)+1-\NV(B)$.
\item[(ii)]
Every blowing-up of $\pi$ is sprouting with respect to $B$ or its preimage.
\item[(iii)]
The exceptional curve of each blowing-up of $\pi$
is a unique ($-1$)-curve in the preimage of $B$.
\end{enumerate}
Conversely,
$[B\TA\TW{n},1]$ shrinks to $B$
for a given positive integer $n$ and a non-empty linear chain $B$.
\end{lemma}

\begin{lemma}\label{lem:bu2}
Let $A$, $B$ be admissible linear chains
and $c$ a positive integer.
Suppose that a composite $\pi$ of blowings-down contracts
$[A,1,B]$ to $[c,1]$.
\begin{enumerate}
\item[(i)]
The first curve of $[c,1]$ is the image of the first curve of $A$.
We have $n:=\NV(A)-\NV(B^{\ast})\ge0$ and $A=[c,\TW{n}]\TA B^{\ast}$.
In particular, $n=0$ if $c=1$.
\item[(ii)]
The first $n$ blowings-up of $\pi$ are sprouting and
the remaining ones are subdivisional with respect to
$[c,1]$ or its preimages.
The composite of the subdivisional blowings-up contracts
$[A,1,B]$ to $[c,\TW{n},1]$.
\item[(iii)]
The exceptional curve of each blowing-up of $\pi$
is a unique ($-1$)-curve in the preimage of $[c,1]$.
\end{enumerate}
\end{lemma}
\begin{proof}
Write $A=[a_1,\ldots,a_r]$, $B=[b_1,\ldots,b_s]$.
We prove the assertions by induction on $r+s\ge 2$.
After the first blowing-down of $\pi$,
$[A,1,B]$ becomes $T:=[\underline{A},a_r-1,b_1-1,\overline{B}]$.
The last blowing-up of $\pi$ satisfies (iii)
and is subdivisional with respect to $T$.
Suppose $r+s=2$.
We have $T=[c,1]$,
$\underline{A}=\overline{B}=\emptyset$, $b_1=2$ and $c=a_r-1$.
By Lemma~\ref{lem:adj}, we obtain $B^{\ast}=[2]$ and $n=0$.
Hence $A=[c]\TA\TW{1}=[c,\TW{n}]\TA B^{\ast}$.
The remaining assertions are clear in this case.
Assume $r+s\ge 3$.
We have $T\ne[c,1]$.
Since $A$ and $B$ are admissible,
$a_r$ or $b_1$ must be equal to $2$.
If $a_r=b_1=2$,
then $T=[\underline{A},1,1,\overline{B}]$,
which is contracted to $[\ldots,0,\ldots]$ by the second blowing-down.
But the latter linear chain cannot shrink to $[c,1]$.
Hence either $a_r$ or $b_1$ must be greater than $2$.

Case (1): $a_r=2$, $b_1>2$.
If $r=1$, then $[b_s,\ldots,b_2,b_1-1,1]$ shrinks to $[1,c]$.
By Lemma~\ref{lem:bu1},
$[b_s,\ldots,b_2,b_1-1]=[1,c]\TA\TW{s-1}$.
Thus $b_s=1$, which is a contradiction.
Hence $r>1$.
Since $\underline{A}$ is admissible,
we have $\underline{A}=[c,\TW{n'}]\TA[b_1-1,\overline{B}]^{\ast}$
by the induction hypothesis,
where $n'=r-\NV([b_1-1,\overline{B}]^{\ast})-1$.
Hence $A=[c,\TW{n'}]\TA[[b_1-1,\overline{B}]^{\ast},2]$.
By Lemma~\ref{lem:adj},
we obtain
$[[b_1-1,\overline{B}]^{\ast},2]=(\TW{1}\TA[b_1-1,\overline{B}])^{\ast}=B^{\ast}$
and
$\NV([b_1-1,\overline{B}]^{\ast})=\NV(B^{\ast})-1$.
The remaining assertions follow from the induction hypothesis.

Case (2): $a_r>2$, $b_1=2$.
If $s=1$, then $[\underline{A},a_r-1,1]$ shrinks to $[c,1]$.
By Lemma~\ref{lem:bu1},
$[\underline{A},a_r-1]=[c,1]\TA\TW{r-1}=[c,\TW{r-1}]$.
Hence $A=[c,\TW{r-1}]\TA\TW{1}=[c,\TW{r-1}]\TA B^{\ast}$.
The remaining assertions also follow from Lemma~\ref{lem:bu1} in this case.
If $s>1$, then
we have $[\underline{A},a_r-1]=[c,\TW{n'}]\TA(\overline{B})^{\ast}$
by the induction hypothesis,
where $n'=r-\NV((\overline{B})^{\ast})$.
By Lemma~\ref{lem:adj},
we obtain
$A=[c,\TW{n'}]\TA(\overline{B})^{\ast}\TA\TW{1}=[c,\TW{n'}]\TA[2,\overline{B}]^{\ast}=[c,\TW{n'}]\TA B^{\ast}$
and
$\NV((\overline{B})^{\ast})=\NV(B^{\ast})$.
The remaining assertions follow from the induction hypothesis.
\end{proof}

The following corollary to Lemma~\ref{lem:bu2}
describes the process of the contractions
of linear chains in Lemma~\ref{lem:indf} (iii).
\begin{corollary}\label{cor:bu0}
Let $A$ and $B$ be admissible linear chains.
Suppose that a composite $\pi$ of blowings-down contracts
$[A,1,B]$ to $[0]$.
\begin{enumerate}
\item[(i)]
The first blowing-up of $\pi$ is sprouting
with respect to $[0]$
and the remaining ones are subdivisional with respect to preimages of $[0]$.
\item[(ii)]
The exceptional curve of each blowing-up of $\pi$
except the first one
is a unique ($-1$)-curve in the preimage of $[0]$.
\end{enumerate}
\end{corollary}

The next one is a corollary to
Lemma~\ref{lem:indf} (iii), Lemma~\ref{lem:bu1} and Lemma~\ref{lem:bu2}.
It will be used to describe the process of
the resolutions of cusps.
\begin{corollary}\label{cor:bu}
Let $a$ be a positive integer and $A$ an admissible linear chain.
Let $B$ be a linear chain which is empty or admissible.
Assume that
a composite $\pi$ of blowings-down contracts
$[A,1,B]$ to $[a]$
and that $[a]$ is the image of $A$ under $\pi$.
\begin{enumerate}
\item[(i)]
The linear chain $[a]$ is the image of the first curve of $A$.
There exits a positive integer $n$ such that
$A^{\ast}=[B,n+1,\TW{a-1}]$.
Moreover,
$A=[a]\TA\TW{n}\TA B^{\ast}$
if $B\ne\emptyset$.
\item[(ii)]
The first $n$ blowings-up of $\pi$ are sprouting and
the remaining ones are subdivisional with respect to
$[a]$ and its preimages.
The composite of the subdivisional blowings-up contracts
$[A,1,B]$ to $[[a]\TA\TW{n},1]$.
\item[(iii)]
The exceptional curve of each blowing-up of $\pi$
is a unique ($-1$)-curve in the preimage of $[a]$.
\end{enumerate}
Conversely,
$[[a]\TA\TW{n}\TA B^{\ast},1,B]$ shrinks to $[a]$
for given positive integers $a$, $n$ and an admissible linear chain $B$.
\end{corollary}

The following corollary follows from Corollary~\ref{cor:bu} (ii).
\begin{corollary}\label{cor:bu2}
Let the notation and the assumption be as in Corollary~\ref{cor:bu}
and $b$ an integer.
Then $\pi$ contracts $[A,1,B,b]$ to $[a,b-n]$.
The second curve of $[a,b-n]$ is the image of the last curve of $[A,1,B,b]$.
\end{corollary}
\subsection{%
Resolution of a cusp}\label{sec:cres}
Let $C$ be a curve on a smooth surface $V$.
Suppose that $C$ has a cusp $P$.
Let $\sigma:V'\rightarrow V$ be the minimal embedded resolution
of the cusp.
That is,
$\sigma$ is the composite of the shortest sequence
of blowings-up such that
the strict transform $C'$ of $C$ intersects $\sigma^{-1}(P)$ transversally.
Let
\(
V'=V_n\stackrel{\sigma_{n-1}}{\longrightarrow}V_{n-1}
\longrightarrow\cdots\longrightarrow
V_2\stackrel{\sigma_1}{\longrightarrow}
V_1\stackrel{\sigma_0}{\longrightarrow}V_0=V
\)
be the blowings-up of $\sigma$.
The following lemma follows from the assumptions that
$P$ is a cusp and $\sigma$ is minimal.
\begin{lemma}\label{lem:cres0}
For $i\ge1$, the strict transform of $C$ on $V_i$
intersects $(\sigma_0\circ\cdots\circ\sigma_{i-1})^{-1}(P)$
in one point, which is on the exceptional curve of $\sigma_{i-1}$.
The point of intersection is the center of $\sigma_{i}$ if $i<n$.
\end{lemma}

We prove the following lemma.
\begin{lemma}\label{lem:cres}
The following assertions hold (cf. \cite{bk,masa}).
\begin{enumerate}
\item[(i)]
The dual graph of $\sigma^{-1}(C)$ has the following shape,
where
$g\ge1$,
$D_0$ is the exceptional curve of $\sigma_{n-1}$
and
$A_1$ contains the exceptional curve of $\sigma_0$ by definition.
\begin{center}
\includegraphics{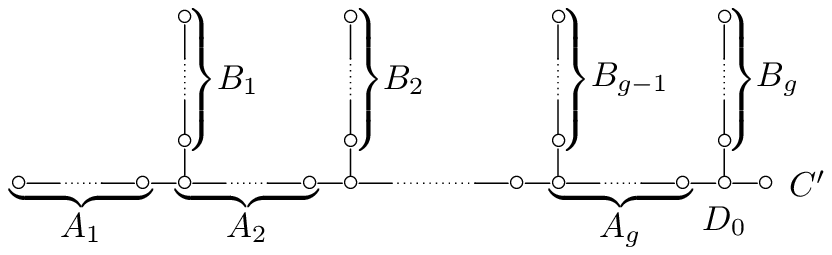}
\end{center}
We number the irreducible components $A_{i,j}$ of $A_i$
(resp.~$B_{i,j}$ of $B_i$)
from the left-hand side to the right
(resp.~the bottom to the top) in the above figure.
\item[(ii)]
The morphism $\sigma$ can be written as
$\sigma=\sigma_0\circ\rho_1'\circ\rho_1''\circ\cdots\circ\rho_g'\circ\rho_g''$,
where each $\rho_i'$ (resp.~$\rho_i''$) consists of
sprouting (resp.~subdivisional) blowings-up
of $\sigma$ with respect to preimages of $P$.
\item[(iii)]
The morphisms $\rho_i:=\rho_i'\circ\rho_i''$
have the following properties.
\begin{enumerate}
\item[(a)]
Each $\rho_i$ maps $A_i$ to a ($-1$)-curve,
which is the image of $A_{i,1}$.
\item[(b)]
$\rho_g$ contracts $A_g+D_0+B_g$ to $A_{g,1}$ and
$\rho_i$ contracts $A_{i}+A_{i+1,1}+B_{i}$ to $A_{i,1}$
for $i<g$.
\end{enumerate}
\end{enumerate}
\end{lemma}
\begin{proof}
For the sake of simplicity,
we do not distinguish between a curve and its strict transforms
via blowings-up.
The second blowing-up of $\sigma$ is sprouting
with respect to the exceptional curve of $\sigma_0$.
Since $P$ is a cusp and $\sigma$ is minimal,
the last blowing-up of $\sigma$ must be subdivisional with respect to
the preimage of $P$.
These facts show the assertion (ii).
Let $E_{0,0}$ denote the exceptional curve of $\sigma_0$ and
$E_{i,0}$ the exceptional curve of
the last blowing-up of $\rho_i''$ for each $i$.
Put $E_0=\emptyset$.
Let $E_i$ denote the exceptional curve of $\rho_i$.
By Lemma~\ref{lem:cres0},
we infer that
the dual graph of the sum of
$E_{i-1,0}$ and
the exceptional curve of $\rho_i'$ is linear.
Hence the dual graph of $E_{i-1,0}+E_i$ is linear.
It follows that
$E_{1,0},\ldots,E_{g-1,0},E_{g,0}=D_0$
are all the branching components of $\sigma^{-1}(C)$.
The divisor $E_{i-1,0}+E_i-E_{i,0}$ consists of
two connected components.
Let $A_i$ denote the one containing $E_{i-1,0}$
and $B_i$ the remaining one.
Then $A_i$, $B_i$ and $\rho_i$
have the desired properties.
\end{proof}

We give the weighted graphs $A_1,\ldots,A_g$ (resp.~$B_1,\ldots,B_g$)
the direction
from the left-hand side to the right
(resp.~from the bottom to the top) of the figure in Lemma~\ref{lem:cres}.
With these directions, we regard $A_i$ and $B_i$ as linear chains.
By Lemma~\ref{lem:cres0},
these linear chains are admissible.
Let $\NS_i$ denote the number of the blowings-up in $\rho_i'$.
The following proposition follows from Corollary~\ref{cor:bu}.
\begin{proposition}\label{prop:cres}
The following assertions hold for $i=1,\ldots,g$.
\begin{enumerate}
\item[(i)]
We have $A_i=\TW{\NS_i}\TA B_i^{\ast}$, $A_i^{\ast}=[B_i,\NS_i+1]$.
\item[(ii)]
The ($-1$)-curve $A_{i,1}$ is the image of the first curve of $A_i$
under $\rho_i$.
\item[(iii)]
The linear chain $A_i$ contains an irreducible component $E$ with $E^2\le-3$.
\end{enumerate}
\end{proposition}

We will use the next lemma to prove 
some properties of the Orevkov's curves.
\begin{lemma}\label{lem:cres2}
Let $D'$ be an SNC-divisor on a smooth surface $V'$.
Suppose the following conditions are satisfied.
\begin{enumerate}
\item[(i)]
The weighted dual graph of $D'$ consists of a ($-1$)-curve $D_0$
and admissible rational linear chains $A_1,B_1,\ldots,A_g,B_g$, $g\ge1$.
They meet each other in the way described in Lemma~\ref{lem:cres} (i).
\item[(ii)]
For $i=1,\ldots,g$,
there exists a positive integer $\NS_i$ such that
$A_i=\TW{\NS_i}\TA B_i^{\ast}$,
or equivalently $A_i^{\ast}=[B_i,\NS_i+1]$.
\end{enumerate}
Then the following assertions hold.
\begin{enumerate}
\item[(a)]
The divisor $D'$ shrinks to a point $P$
by blowings-down $\sigma:V'\rightarrow V$.
The way of blowings-down to contract $D'$ to a point
is unique.
\item[(b)]
Let $C'$ be a smooth curve on $V'$.
If $C'$
intersects only $D_0$ at one point transversally
among the irreducible components of $D'$,
then $\sigma(C')$ is smooth outside of $P$
and has a cusp at $P$, whose
minimal embedded resolution coincides with $\sigma$.
\end{enumerate}
\end{lemma}
\begin{proof}
(a)
By Corollary~\ref{cor:bu},
$[A_g,D_0,B_g]$ shrinks to a ($-1$)-curve,
which is the image of the first curve $A_{g,1}$ of  $A_g$.
The image of
$[A_{g-1},A_{g,1},B_{g-1}]$
under the above contraction
shrinks to a ($-1$)-curve,
which is the image of the first curve of $A_{g-1}$.
Continuing in this way,
we get blowings-down $\sigma:V'\rightarrow V$
which contracts $D'$ to a point $P$.
The uniqueness follows from Corollary~\ref{cor:bu} (iii).

(b)
Since $C'$ is smooth,
$\sigma(C')$ is also smooth outside of $P$.
If the center of a blowing-up of $\sigma$
is not on the image of $C'$,
then those of the remaining blowings-up
are not on the images of $C'$
by Corollary~\ref{cor:bu} (iii).
This contradicts the assumption that
$C'$ intersects $D_0$.
Hence the center of each blowing-up of $\sigma$
is on the image of $C'$.
The remaining assertions of (b) follow from this fact.
\end{proof}
\section{Orevkov's curves and proof of the ``only if'' part of Theorem~1}\label{sec:or}
In this section, we prove some properties of Orevkov's curves,
from which the ``only if'' part of Theorem~\ref{thm1} follows.
In \cite{or},
Orevkov constructed
two sequences $C_{4k}$, $C_{4k}^{\ast}$ ($k=1,2,\ldots$) of
rational unicuspidal plane curves with $\KB=2$ in the following way.
Let $N$ be a nodal cubic.
Let $\Gamma_1$, $\Gamma_2$ denote the two analytic branches
of $N$ at the node.
Let $\phi:W\rightarrow\SP^2$ denote
the composite of $7$-times of blowings-up
such that the center of the first one is the node
and every center of the remaining ones is
the point of intersection
of the strict transform of $\Gamma_1$ and the exceptional curve
of the previous blowing-up.
The dual graph of the exceptional curve $E$ of $\phi$
is connected and linear.
The curve $E$
consists of
$6$-pieces of ($-2$)-curves and one ($-1$)-curve $E'$ as an endpoint
and
intersects the strict transform of $N$ at its two endpoints.

Let $\phi':W\rightarrow\SP^2$ denote
the contraction of the strict transform of $N$ and
the $6$-pieces of ($-2$)-curves in $E$.
Put $f=\phi'\circ\phi^{-1}$.
The curve $\phi'(E')$ is a nodal cubic.
Let $\Gamma$ denote one of the two analytic branches of $\phi'(E')$
at the node such that
the center of the second blowing-up of $\phi'$
is not on its strict transform.
We may assume
$\phi'(E')=N$ and $\Gamma=\Gamma_1$
by composing a suitable projective transformation to $f$.
Let $C_0$ be the tangent line at a flex of $N$ and
$C_0^{\ast}$ an irreducible conic meeting with $N$ only at one smooth point.
See \cite{or,at} or the appendix for the existence of $C_0^{\ast}$.
Orevkov defined $C_{4k}$, $C_{4k}^{\ast}$
as $C_{4k}=f(C_{4k-4})$, $C_{4k}^{\ast}=f(C_{4k-4}^{\ast})$ ($k=1,2,\ldots$).
They have a cusp at the node and tangent to $\Gamma_2$ at the node.
\begin{lemma}\label{lem:oc}
Let $C$ be a rational unicuspidal plane curve,
$\sigma:V\rightarrow\SP^2$ the minimal embedded resolution
of the cusp
and $C'$ the strict transform of $C$ via $\sigma$.
Put $D=\sigma^{-1}(C)$.
Let $A_1$, $B_1,\ldots,A_g,B_g,D_0$ denote the linear chains given
for the cusp by Lemma~\ref{lem:cres}.
\begin{enumerate}
\item[(i)]
The curve $C$ can be constructed in the same way as
$C_{4}$ (resp.~$C_4^{\ast}$)
if and only if $C$ satisfies the following conditions.
\begin{enumerate}
\item[(a)]
$g=1$, $A_1=[\TW{6},4]$, $B_1=\TW{2}$
(resp.~$A_1=[\TW{6},7]$, $B_1=\TW{5}$).
\item[(b)]
There exists a ($-1$)-curve $E_0$
such that
it meets with $D$ at two points transversally
and
intersects only the first curve and the last curve of $A_1$
among the irreducible components of $D$.
\end{enumerate}
\item[(ii)]
The curve $C$ can be constructed in the same way as
$C_{4k+4}$ (resp.~$C_{4k+4}^{\ast}$)
for some $k\ge 1$
if and only if $C$ satisfies the following conditions.
\begin{enumerate}
\item[(a)]
$g=2$, $A_1=\TW{6}^{\ast k+1}$,
$B_1=[7_k]$, $A_2=[4]$, $B_2=\TW{2}$
(resp.~$A_2=[7]$, $B_2=\TW{5}$).
\item[(b)]
There exists a ($-1$)-curve $E_0$
such that
it meets with $D$ at two points transversally
and
intersects only the first curve of $A_1$ and the last curve of $B_1$
among the irreducible components of $D$.
\end{enumerate}
\item[(iii)]
If $C$ can be constructed in the same way as
$C_{4k}$ or $C_{4k}^{\ast}$ for some $k\ge 1$,
then $(C')^2=-2$.
\end{enumerate}
\end{lemma}
\begin{proof}
The assertions for $C_4$ and $C_4^{\ast}$
follow from their definition.
We prove (ii) and (iii) for $C_{4k+4}$, $k\ge1$.
We can similarly deal with $C_{4k+4}^{\ast}$.
We first show the ``if'' part of (ii) by induction on $k$.
Let $a_i$ and $b_i$ denote the $i$-th curves of
the linear chains $A_1$ and $B_1$, respectively.
For the sake of simplicity,
we sometimes use the same symbols for
the strict transforms them via a rational map which does not contract them.

Write $\sigma$ as $\sigma=\sigma_2\circ\sigma_1$,
where $\sigma_2$ consists of seven blowings-up.
By Corollary~\ref{cor:bu} (ii),
the last six blowings-up of $\sigma_2$ are sprouting
with respect to the preimages of the cusp.
The weighted dual graph of
the preimage of the cusp under $\sigma_2$ is the linear chain $[\TW{6},1]$.
By Corollary~\ref{cor:bu} (iii),
the blowings-up of $\sigma_1$ are done over the point of
intersection of $\TW{6}$ and the ($-1$)-curve.
From these facts, we see
$[\TW{6},1]=[\sigma_1(a_1),\ldots,\sigma_1(a_6),\sigma_1(b_k)]$.
The dual graph of $\sigma_1(E_0+a_1+\cdots+a_6+b_k)$ is a loop.
We have $[1,\TW{6},1]=[\sigma_1(E_0),\sigma_1(a_1),\ldots,\sigma_1(a_6),\sigma_1(b_k)]$.
Let $\varphi_1:V_1\rightarrow V_0$
denote the contraction of $\sigma_1(E_0+a_1+\cdots+a_5)$
and $\varphi_0:V_0\rightarrow\SP^2$ the contraction of
$\varphi_1(\sigma_1(a_6))$.
Put $\varphi=\varphi_0\circ\varphi_1$.

We arrange the order of blowings-down of
$\varphi\circ\sigma_1$ in the following way.
We first perform six blowings-down $\varphi_1':V\rightarrow V'$
in the same way as $\varphi_1$.
It contracts $E_0+a_1+\cdots+a_5$ to a point.
Then we perform blowings-down $\sigma_1':V'\rightarrow V_0'$
in the same way as $\sigma_1$.
It contracts
$\varphi_1'(D-(C'+a_1+\cdots+a_6+b_k))$ to a point.
Finally we perform the blowing-down $\varphi_0':V_0'\rightarrow\SP^2$
which contracts $\sigma_1'(\varphi_1'(a_6))$.
The rational map
$\varphi_0'\circ\sigma_1'\circ\varphi_1'\circ(\varphi\circ\sigma_1)^{-1}$
is a projective transformation
since it does not have exceptional curves.
By Corollary~\ref{cor:bu2},
$\varphi_1'(a_6)$ (resp.~$\varphi_1'(b_k)$)
is a ($-2$)-curve (resp.~($-1$)-curve).
The weighted dual graph of
$D-(a_1+\cdots+a_6+b_k)$ is unchanged by $\varphi_1'$.

We decompose the exceptional curve
$\varphi_1'(D-(C'+a_1+\cdots+a_5+b_k))$
of $\varphi_0'\circ\sigma_1'$
into linear chains $A'_1,B'_1,\ldots,A'_{g'},B'_{g'},\varphi_1'(D_0)$.
If $k=1$,
then we set
$g'=1$,
$A_1'=[\varphi_1'(a_6),\ldots,\varphi_1'(a_{11}),\varphi_1'(A_2)]$
and
$B_1'=\varphi_1'(B_2)$.
We have ${(A_1')}^{\ast}=[\TW{6},4]^{\ast}=[B_1',8]$.
If $k>1$,
then we set
$g'=2$,
$A_1'=[\varphi_1'(a_6),\ldots,\varphi_1'(a_{5k+6})]$,
$B'_1=[\varphi_1'(b_1),\ldots,\varphi_1'(b_{k-1})]$,
$A'_2=\varphi_1'(A_2)$
and
$B'_2=\varphi_1'(B_2)$.
We have
${(A_1')}^{\ast}=[7_{k}]=[B_1',7]$.
It follows from Lemma~\ref{lem:cres2}
that $\hat{C}:=\varphi(\sigma_1(C'))$ is unicuspidal
and that $\varphi_0'\circ\sigma_1'$ is the minimal embedded
resolution of the cusp.
The linear chains $A'_1,B'_1,\ldots,A'_{g'},B'_{g'}$ coincide
with those given for $\hat{C}$ by Lemma~\ref{lem:cres}.
By the induction hypothesis ($k>1$) and the assertion (i) ($k=1$),
$\hat{C}$ can be constructed in the same way as $C_{4k}$.
The curve $\varphi_1(\sigma_1(a_6))$
intersects
$\varphi_1(\sigma_1(b_k))$ only at two points transversally.
This shows that $\varphi(\sigma_1(b_k))$ is a nodal cubic.
The morphism $\varphi$ (resp.~$\sigma_2$) performs blowings-up in the same way
as $\phi$ (resp.~$\phi'$).
Thus $C$ can be constructed in the same way as $C_{4k+4}$.

We next show (iii) and the ``only if'' part of (ii).
The curve $C$ is the strict transform of an Orevkov's curve $C_{4k}$
via $f=\phi'\circ\phi^{-1}$.
To avoid confusion,
we denote by
$N_i$
(resp.~$\phi_i:W_i\rightarrow\SP^2$, $\phi_i':W_i\rightarrow\SP^2$)
the nodal cubic $N$ (resp.~the birational morphism $\phi$, $\phi'$)
which is used to make $C_{4i+4}$ from $C_{4i}$ for $i\le k$.
The curve $C_{4k}$ is the strict transform of an Orevkov's curve $C_{4k-4}$
via $f_{k-1}=\phi_{k-1}'\circ\phi_{k-1}^{-1}$.
Let $\sigma:V\rightarrow\SP^2$ denote the minimal embedded resolution
of the cusp of $C$ and $e_i$ the exceptional curve of the $i$-th blowing-up
of $\sigma$.
We note that the strict transform of $N_{k}$ via $\phi_k$
coincides with $e_7$.
Let $\sigma_k:V_k\rightarrow\SP^2$ denote the minimal embedded resolution
of the cusp of  $C_{4k}$.
From the definition of the Orevkov's curves,
we infer that the centers of blowings-up of
$\phi_{k}':W_k\rightarrow\SP^2$
(resp.~$\phi_{k-1}':W_{k-1}\rightarrow\SP^2$)
are the cusp of $C$ (resp.~$C_{4k}$)
and its strict transforms.
This shows that $\sigma:V\rightarrow\SP^2$
(resp.~$\sigma_k:V_k\rightarrow\SP^2$)
can be written as $\sigma=\phi_k'\circ\sigma'$
(resp.~$\sigma_k=\phi_{k-1}'\circ\sigma_k'$),
where $\sigma'$ (resp.~$\sigma_k'$)
consists of blowings-up.

Let $A'_1,B'_1,\ldots,A'_{g'},B'_{g'},D_0'$ denote the linear chains
given by Lemma~\ref{lem:cres} for $C_{4k}$.
If $k=1$, they satisfy the conditions (a), (b) in (i).
Otherwise they satisfy those in (ii) with $k$ being replaced with $k-1$
by the induction hypothesis.
Let $\phi_{k,0}:W_{k,0}\rightarrow\SP^2$
denote the first blowing-up of $\phi_k$,
which coincides with that of $\phi_{k-1}'$.
Let
$\phi_{k,1}$ (resp.~$\phi_{k-1,1}'$) denote
the composite of the remaining blowings-up of $\phi_k$ (resp.~$\phi_{k-1}'$).
Each blowing-up of $\phi_{k,1}$ is done over $\Gamma_1$,
while that of $\phi_{k-1,1}'\circ\sigma_k'$ is done over $\Gamma_2$.
This means that
as a weighted graph,
the strict transform of
$A'_1+B'_1+\cdots+A'_{g'}+B'_{g'}+D_0'$
on $V$
via $\sigma_k^{-1}\circ\phi_{k}\circ\sigma':V\rightarrow V_k$
is obtained by increasing the weight of the first curve of $A_1'$
by one, which is done by the first blowing-up of $\phi_{k,1}$.
Moreover,
$A_1+B_1+\cdots+A_{g}+B_{g}+D_0$ is obtained
by attaching the weighted dual graph of
the strict transform of $e_1+\cdots+e_7$ on $V$
to $\overline{A'}_1+B'_1+A'_2+B'_2+\cdots+A'_{g'}+B'_{g'}+D_0'$.
The first curve of $A_1'$ is replaced with the strict transform of $e_6$.

The curves
$e_6$, $e_7$ and
the strict transform of $C$ on $W_k$
intersect each other in the same way as
$\phi_{k,1}(e_6)$, $\phi_{k,1}(e_7)$ and
the strict transform of $C_{4k}$ on $W_{k,0}$ do.
Furthermore,
$\sigma'$ performs blowings-up in the same way as $\phi_{k-1,1}'\circ\sigma_k'$.
We have $(C')^2=(C_{4k}')^2=-2$
by the induction hypothesis.
The first blowing-up of $\sigma'$ is done at $e_6\cap e_7$
and 
each of the next five blowings-up of $\sigma'$ is done at
the point of intersection of the strict transform of $e_7$
and the exceptional curve of the previous blowing-up.
Let $\sigma'':W_k'\rightarrow W_k$ denote
the composite of the first six blowings-up of $\sigma'$
and $e_i'$, $N_{k+1}'$
the strict transforms of $e_i$, $N_{k+1}$ on $W_k'$, respectively.
The dual graph of $e_{1}'+\cdots+e_{6}'+e_{8}'+\cdots+e_{13}'+e_7'+N_{k+1}'$
is a loop.
We have
$[e_1',\ldots,e_{6}',e_8',\ldots,e_{13}',e_7',N_{k+1}']=[\TW{5},3,\TW{5},1,7,1]$.

As we saw in the proof of the ``if'' part,
if $k>1$, then
$e_{13}'$ is the image of $b_k'$ and
$e_{6}',e_8',\ldots,e_{12}'$ are those of $a_1',\ldots,a_6'$,
respectively,
where $a_i'$ (resp.~$b_i'$) denotes
the strict transform on $V$ of the $i$-th curve of $A_1'$ (resp.~$B_1'$)
via $\sigma_k^{-1}\circ\phi_{k}\circ\sigma':V\rightarrow V_k$.
The remaining blowings-up of $\sigma'$
are done over $e_{12}'\cap e_{13}'$.
It follows from the definition of the Orevkov's curves
that if $k=1$, then
$e_{6}',e_8',\ldots,e_{13}'$
are the images of $a_1',\ldots,a_7'$, respectively.
The remaining blowings-up of $\sigma'$
are done over a point on $e_{13}'\setminus(e_1'+\cdots+e_{12}')$.
Let $e_i''$ denote the strict transform of $e_i$ on $V$.
If $k=1$,
then
$g=2$,
$A_1=[e_{1}',\ldots,e_{6}',\underline{\overline{A_1'}}]=\TW{6}^{\ast 2}$,
$B_1=[e_7'']=[7]$,
$A_2=[e_{13}'']=[a_7']=[4]$ and $B_2=B_1'$.
If $k>1$,
then
$g=2$,
$A_1=[e_{1}',\ldots,e_{6}',\overline{A_1'}]=\TW{6}^{\ast k+1}$,
$B_1=[B_1',e_7'']=[7_k]$,
$A_2=A_2'$ and $B_2=B_2'$.
The strict transform of $N_{k+1}$ via $\sigma$
satisfies the condition that $E_0$ must satisfy.
\end{proof}

By Proposition~\ref{prop:ocpe} below,
each $C_{4k}$ (resp.~$C_{4k}^{\ast}$) does not depend
on the choice of $N$ and $C_0$ (resp.~$C_0^{\ast}$)
up to the projective equivalence.
The ``only if'' part of Theorem~\ref{thm1} follows from
this fact and Lemma~\ref{lem:oc} (iii).
\begin{proposition}\label{prop:ocpe}
Let $C^{(1)}$ and $C^{(2)}$ be plane curves.
If there exists a positive integer $k$ such that
$C^{(1)}$ and $C^{(2)}$ can be constructed in the same way as $C_{4k}$,
or they can be constructed in the same way as $C_{4k}^{\ast}$,
then $C^{(1)}$ is projectively equivalent to $C^{(2)}$.
\end{proposition}
\begin{proof}
We only show the assertion for the case in which
there exists $k\ge2$ such that
$C^{(1)}$ and $C^{(2)}$ can be constructed in the same way as $C_{4k}$.
We can similarly deal with the remaining cases.
For each $i$,
let $\sigma^{(i)}:V^{(i)}\rightarrow\SP^2$
denote the minimal embedded resolution
of the cusp of $C^{(i)}$.
Write $A_1,B_1,\ldots,A_g,B_g$, $D_0$, etc.~given by Lemma~\ref{lem:cres}
for $C^{(i)}$ as
$A_1^{(i)},B_1^{(i)},\ldots,A_{g_i}^{(i)},B_{g_i}^{(i)}$, $D_0^{(i)}$, etc.
Let $E_0^{(i)}$ denote the ($-1$)-curve $E_0$
given for $C^{(i)}$ in Lemma~\ref{lem:oc} (ii).
We define a birational morphism
$\psi^{(i)}:V^{(i)}\rightarrow\SP^2$
in the following way.
It first contracts
$D_0^{(i)}+B_2^{(i)}$
to a point.
Then it contracts the image of
$A_1^{(i)}+E_0^{(i)}+B_1^{(i)}$ to a point.
The last blowing-down of $\psi^{(i)}$
contracts the image $a_1^{(i)}$ of the last curve of $A_1^{(i)}$ to a point.
We infer that $a_1^{(i)}$ intersects
the image of $A_2^{(i)}$ at two points transversally.
It follows that
$\psi^{(i)}(A_2^{(i)})$ is a nodal cubic and that
$\psi^{(i)}({C^{(i)}}')$ is the tangent line at a flex of
$\psi^{(i)}(A_2^{(i)})$.
We may assume that each nodal cubic
$\psi^{(i)}(A_2^{(i)})$ is defined by the equation given in the appendix.
We denote $\psi^{(i)}(A_2^{(i)})$ by $N$.
Let $O_1$, $O_2$ and $O_3$
be the flexes of $N$ defined in the appendix.
There exists a positive integer $a\le3$ such that
$\psi^{(1)}({C^{(1)}}')$ is the tangent line at $O_a$.
Furthermore,
there exists a projective transformation $h$ such that
$h(N)=N$ and $h(\psi^{(1)}({C^{(1)}}'))=\psi^{(2)}({C^{(2)}}')$.

Let $\psi_j^{(i)}:V_j^{(i)}\rightarrow V_{j-1}^{(i)}$
denote the $j$-th blowing-up
of $\psi^{(i)}$, where $V_0^{(i)}=\SP^2$.
Since $h$ maps
the center of $\psi_1^{(1)}$ to that of $\psi_1^{(2)}$,
the rational map
$h_1=\psi_1^{(2)-1}\circ h\circ\psi_1^{(1)}:V_1^{(1)}\rightarrow V_1^{(2)}$
is an isomorphism.
The center of $\psi_2^{(1)}$
is one of the two points of intersection of
$N$ and the exceptional curve of $\psi_1^{(1)}$.
By replacing $h$ with the composite of $h$
and the projective transformation $\varphi_a$ given in the appendix,
if necessary,
we may assume that $h_1$ maps
the center of $\psi_2^{(1)}$ to that of $\psi_2^{(2)}$.
Thus $\psi_2^{(2)-1}\circ h_1\circ\psi_2^{(1)}:V_2^{(1)}\rightarrow V_2^{(2)}$
is an isomorphism.
For the remaining blowings-up,
there are no ambiguities in choices of centers.
It follows that
$h'=\psi^{(2)-1}\circ h\circ\psi^{(1)}:V^{(1)}\rightarrow V^{(2)}$
is an isomorphism.
Since
$h'$ maps the exceptional curve of $\sigma^{(1)}$
to that of $\sigma^{(2)}$,
the rational map $\sigma^{(2)}\circ h'\circ\sigma^{(1)-1}$
is a projective transformation such that
$\sigma^{(2)}\circ h'\circ\sigma^{(1)-1}(C^{(1)})=C^{(2)}$.
\end{proof}
\section{Structure of $\SC^{\ast\ast}$-fibration}
Let $C$ be a rational unicuspidal plane curve
and $P$ the cusp of $C$.
As in Section~\ref{sec:cres},
let $\sigma:V\rightarrow\SP^2$ denote the minimal embedded resolution
of the cusp,
$\sigma_0$ the first blowing-up of $\sigma$
and $C'$ the strict transform of $C$ via $\sigma$.
Put $D=\sigma^{-1}(C)$.
Let $D_0$ denote the exceptional curve of the last blowing-up of $\sigma$.
We decompose the dual graph of $\sigma^{-1}(P)$ into linear chains
$A_1,B_1,\ldots,A_{g},B_{g},D_0$
in the same way as in Section~\ref{sec:cres}.
By Lemma~\ref{lem:cres},
there exists a decomposition
$\sigma=\sigma_0\circ\rho_1'\circ\rho_1''\circ\cdots\circ\rho_g'\circ\rho_g''$,
where each $\rho_i'$ (resp.~$\rho_i''$) consists of
sprouting (resp.~subdivisional) blowings-up
with respect to preimages of $P$.
Let $\NS_i$ denote the number of the blowings-up in $\rho_i'$.

Assume that
the rational unicuspidal plane curve $C$ satisfies the conditions that
$(C')^2=-2$ and $\KB(\SP^2\setminus C)=2$.
We see that one and only one of the two irreducible components of $D-D_0-C'$
meeting with $D_0$ must be a ($-2$)-curve.
Let $F_0'$ denote the ($-2$)-curve and $S_2$ the remaining one.
Let $\varphi_0:V\rightarrow V'$ be the contraction of $D_0$ and $C'$.
Since $(F_0')^2=0$ on $V'$,
there exists a $\SP^1$-fibration $p':V'\rightarrow\SP^1$
such that $F_0'$ is a nonsingular fiber.
Put $p=p'\circ\varphi_0:V\rightarrow\SP^1$.
Since $\KB(\SP^2\setminus C)=2$,
there exists an irreducible component $S_1$ of $D-D_0-F_0'$ meeting with $F_0'$ on $V$.
Put $F_0=F_0'+D_0+C'$.
The surface $X=V\setminus D$ is a $\SQ$-homology plane.
A general fiber of $p|_{X}$ is a curve
$\SC^{\ast\ast}=\SP^1\setminus\{3\text{ points}\}$.
Such fibrations have already been classified in \cite{misu}.
We will use their result to prove our theorem.

There exists a birational morphism $\varphi:V\rightarrow\Sigma_n$
from $V$ onto the Hirzebruch surface $\Sigma_n$ of degree $n$ for some $n$
such that $p\circ\varphi^{-1}:\Sigma_n\rightarrow\SP^1$ is a $\SP^1$-bundle.
The morphism $\varphi$ is the composite of the successive contractions
of the ($-1$)-curves in the singular fibers of $p$.
The curve $S_1$ (resp.~$S_2$) is a 1-section (resp.~2-section) of $p$.
The divisor $D$ contains no other sections of $p$.
\begin{lemma}\label{lem:s}
We may assume that $\varphi(S_1+S_2)$ is smooth.
We have $\varphi(S_1)^2=-1$ and $\varphi(S_2)^2=4$.
\end{lemma}
\begin{proof}
We only prove the first assertion.
Suppose $\varphi(S_1+S_2)$ has a singular point $P$.
Let $\phi_1$ be the blowing-up at $P$.
Since $S_1+S_2$ is smooth on $V$,
we can choose the order of the blowings-up of $\varphi$
such that $\varphi=\phi_1\circ\varphi'$.
Let $F'$ be the strict transform via $\phi_1$
of the fiber of $p\circ\varphi^{-1}$ passing through $P$.
Let $\phi_2$ be the contraction of $F'$.
Since $F'$ is an irreducible component of
a singular fiber of $p\circ{\varphi'}^{-1}$,
we can replace $\varphi$ with $\phi_2\circ\varphi'$.
We infer that $P$ can be resolved by repeating the above process.
Hence we may assume that $\varphi(S_1+S_2)$ is smooth.
\end{proof}

Each singular fiber of $p$ intersects $S_2$ in at most two points.
Suppose that there exists
a singular fiber $F_2$ of $p$ meeting with $S_2$ in two points.
Let $E_2$ be the sum of the irreducible components of $F_2$
which are not components of $D$.
Because $D$ contains no loop, $E_2$ is not empty.
Since $\KB(V\setminus D)=2$,
each irreducible component of $E_2$ meets with $D$ in at least two points
by \cite[Main Theorem]{mits2}.
In \cite[Lemma 1.6]{misu},
singular fibers of a $\SC^{\ast\ast}$-fibration
with a 2-section
were classified into several types.
Among them,
only singular fibers
of type  ($\mathrm{I_1}$) and ($\mathrm{III_1}$)
satisfy the conditions that
they meet with the 2-section in two points
and that
each irreducible component of $E_2$ meets with $D$ in at least two points.
From the fact that $D$ contains no loop,
we infer that
$F_2$ is of type ($\mathrm{III_1}$).
The dual graph of $F_2+S_1+S_2$ coincides with one of those in the following figure,
where $\ast$ denotes a ($-1$)-curve and $E_2=E_{21}+E_{22}$.
The divisor $T_{2,i}$ may be empty for each $i$.
\begin{center}
\includegraphics{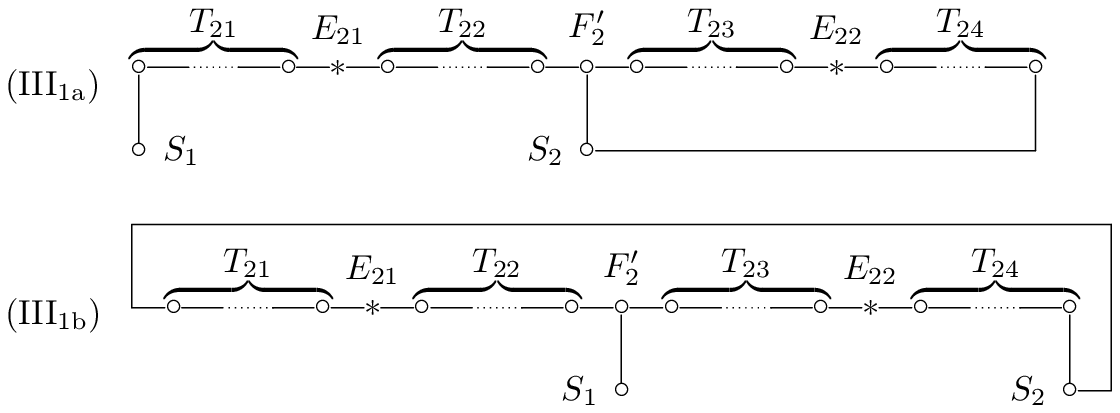}
\end{center}
\begin{lemma}
We have $\varphi(F_2)=\varphi(F_2')$,
where $F_2'$ is the irreducible component of $F_2$
whose position in $F_2$ is illustrated in the above figure.
\end{lemma}
\begin{proof}
Suppose that $\varphi$ contracts $F_2'$.
Write $\varphi=\phi_3\circ\phi_2\circ\phi_1$,
where $\phi_2$ is the contraction of $F_2'$.
If $F_2$ is of type ($\mathrm{III_{1a}}$),
then $\phi_1(F_2')\phi_1(S_1)=0$
and $\phi_1(F_2')\phi_1(S_2)=1$ by Lemma~\ref{lem:s}.
Since $\phi_1(F_2-F_2')\phi_1(F_2')\ge2$,
we have $\phi_2(\phi_1(F_2))\phi_2(\phi_1(S_2))\ge3$,
which is a contradiction.
If $F_2$ is of type ($\mathrm{III_{1b}}$),
then $\phi_1(F_2')\phi_1(S_2)=0$ by Lemma~\ref{lem:s}.
We have $\phi_2(\phi_1(F_2))\phi_2(\phi_1(S_1))\ge2$,
which is absurd.
\end{proof}

Suppose that there exists
a singular fiber $F_1$ of $p$ which intersects $S_2$ in one point.
Let $E_1$ be the sum of the irreducible components of $F_1$
which are not components of $D$.
By the same reasoning as for $F_2$,
we deduce that
$F_1$ is of type ($\mathrm{IV_2}$).
See \cite[Lemma 1.6]{misu}.
The dual graph of $F_1+S_1+S_2$ coincides with
one of those in the following figure,
where $\bullet$ denotes a ($-2$)-curve.
The divisor $T_{1,i}$ may be empty for each $i$.
\begin{center}
\includegraphics{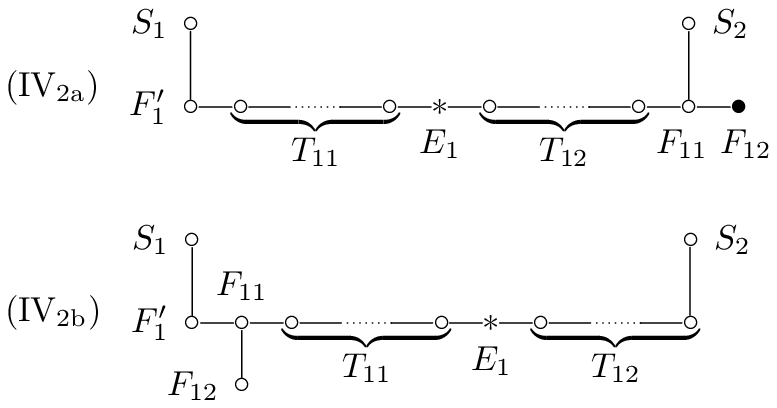}
\end{center}

We can choose the order of the blowings-down of $\varphi$ such that
$\varphi=\varphi'\circ\varphi_1\circ\varphi''$, where
$\varphi_1$ is the composite of all the contractions of irreducible components of $F_1$.
\begin{lemma}\label{lem:f1}
The morphism
$\varphi_1$ contracts
$\varphi''(T_{11}+E_1+T_{12}+F_{11})$ to a ($-1$)-curve,
which is the image of $F_{11}$,
and then contracts the ($-1$)-curve and the image of $F_{12}$
in this order.
We have $\varphi(F_1)=\varphi(F_1')$.
Moreover, $(F_1')^2=F_{12}^2=-2$ if $F_1$ is of type $\mathrm{(IV_{2b})}$.
\end{lemma}
\begin{proof}
Suppose that $F_1$ is of type ($\mathrm{IV_{2b}}$).
Since $(F_1')^2\le -2$, $F_{12}^2\le-2$,
$\varphi$ contracts
$F_{11}$ before the contractions of $F_1'$ and $F_{12}$.
Since $\varphi(F_1)$ is smooth,
$T_{11}+E_1+T_{12}$ must be contracted to a point
before the contraction of $F_{11}$.
It follows that $(F_1')^2=F_{12}^2=-2$.
By Lemma~\ref{lem:s},
$\varphi$ does not contract $F_1'$.

Suppose that $F_1$ is of type ($\mathrm{IV_{2a}}$).
Assume $\varphi$ contracts $F_1'$.
By Corollary~\ref{cor:bu0},
$F_{1}'$ is the exceptional curve of the first blowing-up of $\varphi_1$.
The remaining blowings-up are subdivisional with respect to
the preimages of $\varphi_1(\varphi''(F_1))$.
By Lemma~\ref{lem:s},
the center of the first blowing-up is not on $\varphi_1(\varphi''(S_2))$.
This means that $F_{1}S_2=2$, which is a contradiction.
Thus $\varphi$ does not contract $F_1'$.
By Corollary~\ref{cor:bu0},
$F_{12}$ is the exceptional curve of the first blowing-up of $\varphi_1$.
Since the remaining blowings-up are subdivisional with respect to
the preimages of $\varphi_1(\varphi''(F_1))$,
we infer that
the exceptional curve of the second blowing-up of $\varphi_1$ coincides with
the image of $F_{11}$.
\end{proof}

By the Riemann-Hurwitz formula,
$p$ has no more than two singular fibers which meet with $S_2$ in one point.
By \cite[Lemma 2.3]{misu}, 
$p$ has one singular fiber of type ($\mathrm{III_{1}}$).
It follows
that the dual graph of D must be one of those
in Figure~\ref{fig}.
\begin{figure}
\begin{center}
\includegraphics{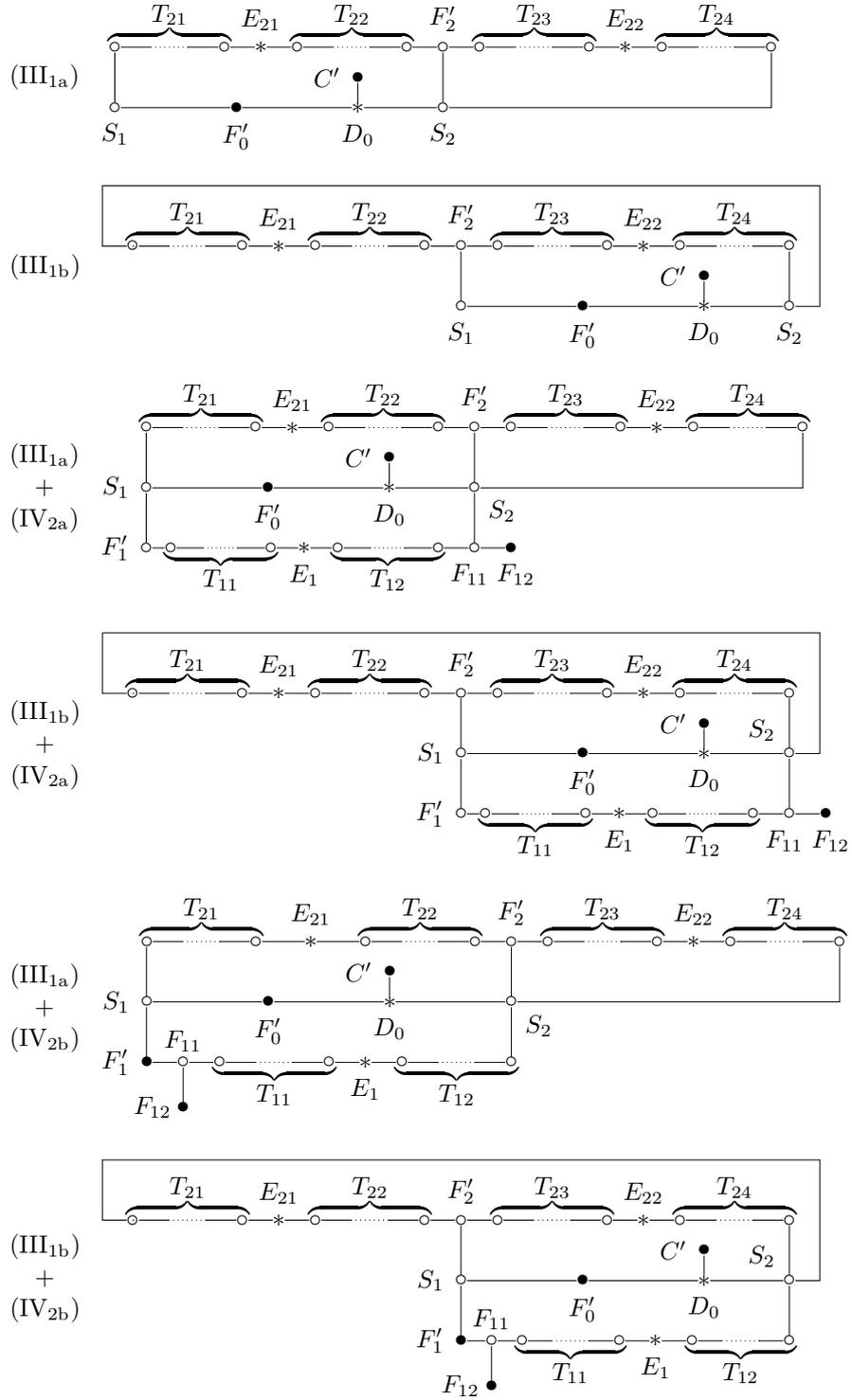}
\caption{Dual graphs of $S_1+S_2+F_0+F_1+F_2$}\label{fig}
\end{center}
\end{figure}
\section{Proof of the ``if'' part of Theorem~1 and Theorem~2}\label{sec:pf}
%
%
%
Let the notation be as in the previous section.
We determine which graphs in Figure~\ref{fig} can be realized.
With the direction from the left-hand side to the right of Figure~\ref{fig},
we regard $T_{ij}$'s as linear chains.
Put $s_i=-S_i^2$ and $f_i=-(F_i')^2$ for each $i$.
We have $s_2\ge 3$, $s_1\ge 2$ and $f_i\ge 2$ for each $i$.

$\mathrm{\bf(III_{1a})}$.
We may assume $\varphi=\varphi_0\circ\varphi_{21}\circ\varphi_{22}$,
where
$\varphi_{22}$
(resp.~$\varphi_{21}$, $\varphi_0$) contracts
$T_{23}+E_{22}+T_{24}$
(resp.~$\varphi_{22}(T_{21}+E_{21}+T_{22})$,
$\varphi_{21}(\varphi_{22}(C'+D_0))$) to a point.
We first show the following lemma.
\begin{lemma}\label{lem:31a}
There exist positive integers $k_{12}$ and $k_{34}$ such that
$[S_1,T_{21}]^{\ast}=[T_{22},k_{12}+1]$ and
$[F_2',T_{23}]^{\ast}=[T_{24},k_{34}+1,\TW{k_{12}-1}]$.
We have
$k_{34}=s_2+2\ge 5$, $T_{23}\ne\emptyset$,
$B_g=[F_0',S_1,T_{21}]$ and $A_g=\TW{\NS_g}\TA[T_{22},k_{12}+2]$.
\end{lemma}
\begin{proof}
By Lemma~\ref{lem:s}, $\varphi_{21}(\varphi_{22}(S_1))$ is a ($-1$)-curve.
The morphism $\varphi_{22}$ does not change
the linear chain $[S_1,T_{21},E_{21},T_{22}]$.
We apply Corollary~\ref{cor:bu}
to $[S_1,T_{21},E_{21},T_{22}]$ and $\varphi_{21}$.
There exists a positive integer $k_{12}$ such that
$[S_1,T_{21}]^{\ast}=[T_{22},k_{12}+1]$.
Since $\varphi_{21}(\varphi_{22}(F_2'))$ is a $0$-curve,
$\varphi_{22}(F_2')$ must be a ($-k_{12}$)-curve
by Corollary~\ref{cor:bu2}.
Again by Corollary~\ref{cor:bu},
there exists a positive integer $k_{34}$ such that
$[F_2',T_{23}]^{\ast}=[T_{24},k_{34}+1,\TW{k_{12}-1}]$.
Since $\varphi(S_2)^2=4$,
we have $4=-s_2+k_{34}+2$
by Corollary~\ref{cor:bu2}.
If $T_{23}=\emptyset$,
then $[T_{24},k_{34}+1,\TW{k_{12}-1}]=\TW{f_2-1}$
by Lemma~\ref{lem:adj}.
We have $k_{34}=1$.
Thus $s_2=-1$, which is absurd.
Hence $T_{23}\ne\emptyset$.
Either
$A_g=\TP{[F_0',S_1,T_{21}]}$ or $B_g=[F_0',S_1,T_{21}]$
by Lemma~\ref{lem:cres} (i).
Suppose the former case holds.
We have $g=1$.
Since $T_{23}\ne\emptyset$,
we see $B_1=[S_2,F_2',T_{23}]$ and $T_{22}=\emptyset$.
By Proposition~\ref{prop:cres} and Lemma~\ref{lem:adj},
$[\NS_1+1,\TP{B_1}]=\TP{A_1}^{\ast}=[F_0',S_1,T_{21}]^{\ast}=[S_1,T_{21}]^{\ast}\TA\TW{1}=[k_{12}+2]$,
which is a contradiction.
Thus $B_g=[F_0',S_1,T_{21}]$.
By Proposition~\ref{prop:cres} and Lemma~\ref{lem:adj},
$A_g=\TW{\NS_g}\TA B_g^{\ast}=\TW{\NS_g}\TA [S_1,T_{21}]^{\ast}\TA\TW{1}=\TW{\NS_g}\TA[T_{22},k_{12}+2]$.
\end{proof}
Case (i): $T_{24}=\emptyset$.
By Lemma~\ref{lem:31a},
$[F_2',T_{23}]=[k_{34}+1,\TW{k_{12}-1}]^{\ast}$.
By Lemma~\ref{lem:adj},
$[k_{34}+1,\TW{k_{12}-1}]^{\ast}=[k_{12}+1,\TW{k_{34}-1}]$.
Thus $f_2=k_{12}+1$ and $T_{23}=\TW{k_{34}-1}$.
Suppose $T_{22}\ne\emptyset$.
We have $g=2$ and $A_2=[F_2',S_2]$ by Lemma~\ref{lem:cres} (i).
By Lemma~\ref{lem:31a},
we obtain $\NS_2=1$, $[f_2-1]=T_{22}$, $s_2=k_{12}+2$ and $k_{34}=k_{12}+4$.
Either $T_{23}=\TP{A}_1$ or $T_{23}=B_1$.
Since $T_{23}$ consists of ($-2$)-curves,
it follows from Proposition~\ref{prop:cres} (iii) that
$T_{23}=B_1$ and $T_{22}=A_1$.
By Proposition~\ref{prop:cres},
$T_{22}=A_1=\TW{\NS_1}\TA B_1^{\ast}=\TW{\NS_1}\TA T_{23}^{\ast}=\TW{\NS_1}\TA [k_{12}+4]$.
Thus $\NS_1=1$ and $f_2=k_{12}+6$, which contradicts $f_2=k_{12}+1$.
Hence $T_{22}=\emptyset$.
We have $g=1$.
By Lemma~\ref{lem:31a},
$[S_1,T_{21}]=\TW{k_{12}}$.
This means that
$A_1=\TP{[S_2,F_2',T_{23}]}$ and $B_1=[S_1,T_{21}]$.
By Lemma~\ref{lem:31a},
$[\TW{k_{34}-1},f_2,s_2]=\TW{\NS_1}\TA[k_{12}+2]$.
We see
$s_2=k_{12}+3$, $f_2=2$ and $\NS_1=k_{34}+1$.
It follows that $k_{12}=1$, $s_2=4$, $k_{34}=6$ and $\NS_1=7$.
We have $A_1=[\TW{6},4]$ and $[B_1,\NS_1+1]=A_1^{\ast}=[\TW{2},8]$.
The curve $E_{22}$ intersects only the first and the last curve of $A_1$
among the irreducible components of $D$.
By Lemma~\ref{lem:oc}, $C$ can be constructed as $C_4$.

Case (ii): $T_{24}\ne\emptyset$.
Since $S_2$ is a branching component of $D$,
we infer $A_{g}=S_2$ by Lemma~\ref{lem:cres} (i).
By Lemma~\ref{lem:31a},
we obtain $\NS_g=1$, $T_{22}=\emptyset$, $s_2=k_{12}+3$ and $k_{34}=k_{12}+5$.
We have $g=2$.
Either $B_1=[F_2',T_{23}]$ or $B_1=\TP{T}_{24}$.
If $B_1=[F_2',T_{23}]$,
then $T_{24}=A_1=\TW{\NS_1}\TA[F_2',T_{23}]^{\ast}=\TW{\NS_1}\TA[T_{24},k_{34}+1,\TW{k_{12}-1}]$, which is impossible.
Thus
$B_1=\TP{T}_{24}$ and $A_1=\TP{[F_2',T_{23}]}$.
By Proposition~\ref{prop:cres},
$[\NS_1+1,T_{24}]=\TP{A}_1^{\ast}=[F_2',T_{23}]^{\ast}$.
By Lemma~\ref{lem:31a},
$[\NS_1+1,T_{24}]=[T_{24},k_{12}+6,\TW{k_{12}-1}]$.
Hence $k_{12}=1$, $[\NS_1+1,T_{24}]=[T_{24},7]$.
It follows from Lemma~\ref{lem:adj} that
$\NS_1=6$ and $T_{24}=[7_k]$, where $k=\NV(T_{24})\ge1$.
We have $B_1=[7_k]$, $A_1=\TW{\NS_1}\TA B_1^{\ast}=\TW{6}^{\ast k+1}$
and $A_2=[4]$.
Since $[B_2,\NS_2+1]=A_2^{\ast}=\TW{3}$, we obtain $B_2=\TW{2}$.
The curve $E_{22}$ intersects only the first curve of $A_1$ and
the last curve of $B_1$
among the irreducible components of $D$.
By Lemma~\ref{lem:oc}, $C$ can be constructed as $C_{4k+4}$.

$\mathrm{\bf(III_{1a})+(IV_{2a})}$.
We have $A_g=S_2$ and $B_g=[F_0',S_1,F_1',T_{11}]$
because $S_2$ is a branching component of $D$.
By Proposition~\ref{prop:cres},
$[B_g,\NS_g+1]=A_g^{\ast}=\TW{s_2-1}$.
Thus $[F_1',T_{11}]=\TW{s_2-4}$.
By Lemma~\ref{lem:f1},
$\varphi$ contracts $F_1$ to a $0$-curve,
which is the image of $F_1'$.
By Lemma~\ref{lem:indf} (iii),
$[T_{12},F_{11},F_{12}]=[F_1',T_{11}]^{\ast}=\TW{s_2-4}^{\ast}=[s_2-3]$,
which is absurd.
Hence this case does not occur.

$\mathrm{\bf(III_{1a})+(IV_{2b})}$.
We may assume
$\varphi=\varphi_0\circ\varphi_1\circ\varphi_{21}\circ\varphi_{22}$,
where
$\varphi_{22}$
(resp.~$\varphi_{21}$, $\varphi_1$, $\varphi_0$) contracts
$T_{23}+E_{22}+T_{24}$
(resp.~$\varphi_{22}(T_{21}+E_{21}+T_{22})$,
$\varphi_{21}(\varphi_{22}(F_{11}+F_{12}+T_{11}+E_{11}+T_{12}))$,
$\varphi_{1}(\varphi_{21}(\varphi_{22}(C'+D_0)))$) to a point.
We show the following three lemmas.
\begin{lemma}\label{lem:31a42b1}
There exist positive integers $k_{12}$ and $k_{34}$ such that
$[S_1,T_{21}]^{\ast}=[T_{22},k_{12}+1]$ and
$[F_2',T_{23}]^{\ast}=[T_{24},k_{34}+1,\TW{k_{12}-1}]$.
We have
$[F_{11},T_{11}]^{\ast}=[T_{12},s_2-k_{34}+1]$
and $s_2\ge k_{34}+1$.
\end{lemma}
\begin{proof}
By the same arguments as in the proof of Lemma~\ref{lem:31a},
there exist positive integers $k_{12}$, $k_{34}$ such that
$[S_1,T_{21}]^{\ast}=[T_{22},k_{12}+1]$ and
$[F_2',T_{23}]^{\ast}=[T_{24},k_{34}+1,\TW{k_{12}-1}]$.
By Lemma~\ref{lem:f1} and Corollary~\ref{cor:bu},
there exists a positive integer $l$ such that
$[F_{11},T_{11}]^{\ast}=[T_{12},l+1]$.
Since $\varphi(S_2)^2=4$,
we infer $4=-s_2+k_{34}+2+l+2$.
Thus $1\le l=s_2-k_{34}$.
\end{proof}
\begin{lemma}\label{lem:31a42b2}
We have $T_{21}=\emptyset$,
$T_{22}=\TW{s_1-2}$ and
$k_{12}=1$.
\end{lemma}
\begin{proof}
Suppose that $S_1$ is a branching component of $D$.
We have $A_g=[S_1,F_0']$,
$T_{12}=T_{24}=\emptyset$ and $B_g=[S_2,F_2',\ldots]$.
By Lemma~\ref{lem:31a42b1},
$[F_{11},T_{11}]=\TW{s_2-k_{34}}$ and
$[F_2',T_{23}]=[k_{12}+1,\TW{k_{34}-1}]$.
By Proposition~\ref{prop:cres},
$[B_g,\NS_g+1]=A_g^{\ast}=\TW{1}\TA\TW{s_1-1}=[3,\TW{s_1-2}]$.
Thus $\NS_g=1$, $f_2=2$ and $s_2=3$.
Since $f_2=k_{12}+1$,
we obtain $k_{12}=1$.
Because $\emptyset\ne [F_{11},T_{11}]=\TW{3-k_{34}}$,
we have $k_{34}\le 2$.
If $k_{34}=1$, then $T_{23}=\TW{k_{34}-1}=\emptyset$.
Thus $B_g=[S_2,F_2',\TP{T}_{22}]$.
By Proposition~\ref{prop:cres},
$A_g=\TW{\NS_g}\TA B_g^{\ast}=\TW{1}\TA[3,2,\TP{T}_{22}]^{\ast}=\TW{1}\TA[2,\TP{T}_{22}]^{\ast}\TA\TW{2}$.
By Lemma~\ref{lem:31a42b1},
$\TW{1}\TA[2,\TP{T}_{22}]^{\ast}\TA\TW{2}=\TW{1}\TA[\TP{T}_{21},S_1]\TA\TW{2}$.
This means that $S_1=\TW{1}\TA[\TP{T}_{21},S_1]\TA\TW{1}$, which is impossible.
Hence $k_{34}=2$.
Since $T_{23}=[2]\ne\emptyset$,
we infer $B_g=[S_2,F_2',T_{23}]$ and $T_{22}=\emptyset$.
By Lemma~\ref{lem:31a42b1}, $[S_1,T_{21}]=\TW{k_{12}}=[2]$, which is absurd.
Hence $S_1$ is not a branching component of $D$.
We have $T_{21}=\emptyset$.
By Lemma~\ref{lem:31a42b1}, $[T_{22},k_{12}+1]=\TW{s_1-1}$.
From this, we obtain $k_{12}=1$ and $T_{22}=\TW{s_1-2}$.
\end{proof}
\begin{lemma}\label{lem:31a42b3}
We have $T_{11}=T_{12}=\emptyset$,
$B_g=[F_0',S_1,F_1',F_{11},F_{12}]$,
$s_2=k_{34}+1$
and
$F_{11}=[2]$.
\end{lemma}
\begin{proof}
Either $S_2\subset A_g$ or $S_2\subset B_g$.
Suppose $S_2\subset B_g$.
We have $T_{24}=T_{12}=\emptyset$.
By Lemma~\ref{lem:31a42b1}, $[F_2',T_{23}]=[k_{34}+1]^{\ast}=\TW{k_{34}}$.
Thus $f_2=2$, $T_{23}=\TW{k_{34}-1}$.
Since $[F_{11},T_{11}]=\TW{s_2-k_{34}}$,
we get $F_{11}=[2]$ and $T_{11}=\TW{s_2-k_{34}-1}$.
If $T_{11}\ne\emptyset$,
then $A_1=F_{12}$ or $A_1=\TP{T}_{11}$ since $F_{11}$ is a branching component of $D$.
Thus $A_1$ consists of ($-2$)-curves,
which contradicts Proposition~\ref{prop:cres}.
Hence $T_{11}=\emptyset$.
We have $s_2=k_{34}+1$, $g=1$ and $A_1=[F_{12},F_{11},F_1',S_1,F_0']=[\TW{3},S_1,2]$.
We infer $s_1\ge 3$.
By Proposition~\ref{prop:cres},
$[B_1,\NS_1+1]=A_1^{\ast}=[3,\TW{s_1-3},5]$.
This means that $s_2=3$ and $k_{34}=2$.
Since $T_{23}=[2]\ne\emptyset$,
we have $B_1=[S_2,F_2',T_{23}]$ and $T_{22}=\emptyset$.
By Lemma~\ref{lem:31a42b2}, $s_1=2$, which is a contradiction.
Hence $S_2\subset A_g$.
We have $B_g=[F_0',S_1,F_1',F_{11},F_{12}]$ and $T_{11}=\emptyset$.
By Lemma~\ref{lem:31a42b1}, $[T_{12},s_2-k_{34}+1]=\TW{-F_{11}^2-1}$.
This shows $s_2=k_{34}+1$ and $T_{12}=\TW{-F_{11}^2-2}$.
If $T_{12}\ne\emptyset$, then $F_{11}^2<-2$ and $A_g=S_2$.
By Proposition~\ref{prop:cres},
$[B_g,\NS_g+1]=A_{g}^{\ast}=\TW{s_2-1}$, which is absurd.
Hence $T_{12}=\emptyset$ and $F_{11}=[2]$.
\end{proof}

Case (i): $T_{24}=\emptyset$.
By Lemma~\ref{lem:31a42b1},
$[F_2',T_{23}]=\TW{k_{34}}$.
We have $f_2=2$ and $T_{23}=\TW{k_{34}-1}=\TW{s_2-2}\ne\emptyset$.
If $T_{22}\ne\emptyset$, then $A_1=T_{22}$ or $A_1=\TP{T}_{23}$.
Thus $A_1$ consists of ($-2$)-curves,
which contradicts Proposition~\ref{prop:cres}.
Hence $T_{22}=\emptyset$.
We infer $g=1$ and $A_1=\TP{[S_2,F_2',T_{23}]}=[\TW{k_{34}},k_{34}+1]$.
By Lemma~\ref{lem:31a42b2},
we have $S_1=[2]$ and $B_1=\TW{5}$.
By Proposition~\ref{prop:cres}, $A_1=\TW{\NS_1}\TA[6]=[\TW{\NS_1-1},7]$.
Hence $k_{34}=6$, $A_1=[\TW{6},7]$.
The curve $E_{22}$ intersects only the first and the last curve of $A_1$
among the irreducible components of $D$.
By Lemma~\ref{lem:oc}, $C$ can be constructed as $C_4^{\ast}$.

Case (ii): $T_{24}\ne\emptyset$.
We have $A_g=S_2$.
By Proposition~\ref{prop:cres},
$[B_g,\NS_g+1]=A_g^{\ast}=\TW{s_2-1}$.
We see $S_1=[2]$, $B_g=\TW{5}$, $s_2=7$ and $k_{34}=6$
by Lemma~\ref{lem:31a42b3}.
By Lemma~\ref{lem:31a42b2}, $T_{22}=\emptyset$.
We infer $g=2$.
Either $B_1=\TP{T}_{24}$ or $A_1=T_{24}$.
If $A_1=T_{24}$, then $B_1=[F_2',T_{23}]$.
By Proposition~\ref{prop:cres} and Lemma~\ref{lem:31a42b1},
$T_{24}=\TW{\NS_1}\TA [F_2',T_{23}]^{\ast}=\TW{\NS_1}\TA [T_{24},7]$, which is absurd.
Hence $B_1=\TP{T}_{24}$ and $A_1=\TP{[F_2',T_{23}]}$.
By Proposition~\ref{prop:cres} and Lemma~\ref{lem:31a42b1},
$[\NS_1+1,T_{24}]=[F_2',T_{23}]^{\ast}=[T_{24},7]$.
It follows from Lemma~\ref{lem:adj} that
$\NS_1=6$, $T_{24}=[7_k]$, where $k=\NV(T_{24})\ge1$.
We have $B_2=\TW{5}$, $A_2=[7]$, $B_1=[7_k]$ and $A_1=\TW{6}^{\ast k+1}$.
The curve $E_{22}$ intersects only the first curve of $A_1$ and
the last curve of $B_1$
among the irreducible components of $D$.
By Lemma~\ref{lem:oc}, $C$ can be constructed as $C_{4k+4}^{\ast}$.

$\mathrm{\bf(III_{1b})}$,
$\mathrm{\bf(III_{1b})+(IV_{2a})}$
or
$\mathrm{\bf(III_{1b})+(IV_{2b})}$.
In each case,
we have $-2\ge\varphi(S_1)^2$ because $S_1$ meets with
only $F_i'$ among the irreducible components of $F_i$ for each $i$.
Hence all the cases do not occur.
\begin{figure}[tb]
\begin{center}
\includegraphics{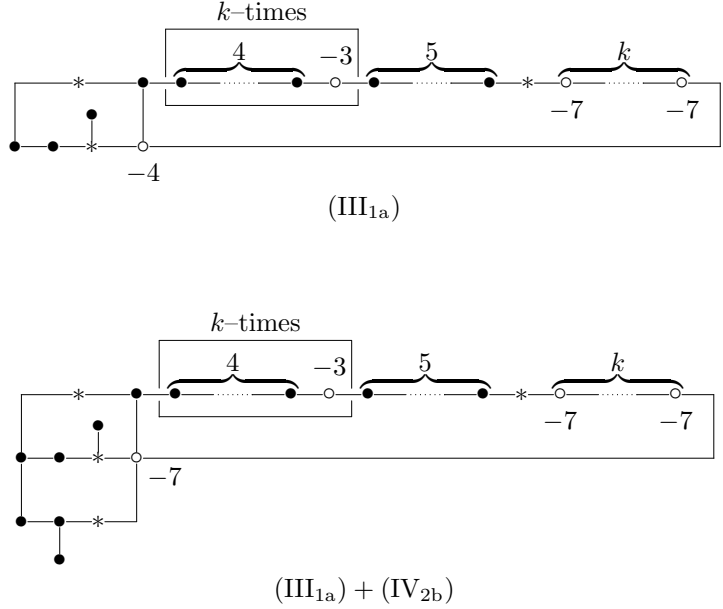}
\caption{The dual graphs of $D+E_1+E_2$}\label{fig4}
\end{center}
\end{figure}

We list the weighted dual graphs of  $D+E_1+E_2$ in Figure~\ref{fig4},
where $k=0$ if $T_{24}=\emptyset$.
We proved that if a rational unicuspidal plane curve $C$
satisfies the conditions $(C')^2=-2$, $\KB=2$,
then $C$ can be constructed in the same way
as $C_{4k}$ or $C_{4k}^{\ast}$ for some $k$.
By Proposition~\ref{prop:ocpe},
$C$ is projectively equivalent to $C_{4k}$ or $C_{4k}^{\ast}$.
We have thus proved Theorem~\ref{thm1}.

Finally, we prove Theorem~\ref{thm:pm}.
The ``only if'' part of Theorem~\ref{thm:pm}
follows from \cite[Lemma 4.4]{ko2}.
We show the ``if'' part.
By Theorem 3.1, Lemma 4.2, Lemma 4.5 and Lemma 4.6 of \cite{ko2},
we deduce that
if $\KB(\SP^2\setminus C)\ge 0$,
$\PG_2(\SP^2\setminus C)=\PG_3(\SP^2\setminus C)=0$,
then $C$ is a rational unicuspidal curve such that $\KB=2$, $(C')^2=-2$ and
the dual graph of the exceptional curve
of the minimal embedded resolution of $C$ is linear.
Thus the ``if'' part follows from Theorem~\ref{thm1} and  Lemma~\ref{lem:oc}.
\begin{acknowledgment}
The author would like to express his thanks to Professor Fumio Sakai
for his helpful advice.
\end{acknowledgment}
\section*{Appendix by Fumio Sakai}
Let $N$ be the nodal cubic $x^3+y^3-xyz=0$.
Let $O$ denote the node $(0,0,1)$.
It is well known that the set $N\setminus \{O\}$ has a group structure,
which is isomorphic to the multiplicative group $\SC^{\ast}$.
The group isomorphism is given by
$\phi:\SC^{\ast}\ni t \mapsto (t,-t^2,t^3-1)\in N\setminus \{O\}$.
Geometrically, we have $t_1t_2t_3=1$ if and only if
$\phi(t_1)$, $\phi(t_2)$ and $\phi(t_3)$ are collinear.
We see easily that $N$ has three flexes
$O_1=(1,-1,0)=\phi(1)$, $O_2=(1,-\omega,0)=\phi(\omega)$ and
$O_3=(1,-\omega^2,0)=\phi(\omega^2)$, where $\omega=e^{2\pi i/3}$.
There exist three projective transformations
\[
\varphi_1=\left(
\begin{array}{ccc}
0 & 1 & 0\\
1 & 0 & 0\\
0 & 0 & 1
\end{array}
\right),\quad
\varphi_2=\left(
\begin{array}{ccc}
0 & \omega^2 & 0\\
\omega & 0 & 0\\
0 & 0 & 1
\end{array}
\right),\quad
\varphi_3=\left(
\begin{array}{ccc}
0 & \omega & 0\\
\omega^2 & 0 & 0\\
0 & 0 & 1
\end{array}
\right)
\]
such that $\varphi_i(O_i)=O_i$, $\varphi_i(O_j)=O_k$ for distinct
$i$, $j$, $k$ among $\{1,2,3\}$.
\begin{theorem}
Define three conics
\begin{eqnarray*}
&&Q_1:21(x^2+y^2)-22xy-6(x+y)z+z^2=0,\\
&&Q_2:21(\omega x^2+\omega^2y^2)-22xy-6(\omega^2x+\omega y)z+z^2=0,\\
&&Q_3:21(\omega^2x^2+\omega y^2)-22xy-6(\omega x+\omega^2y)z+z^2=0.
\end{eqnarray*}
Then the conic $Q_1$ (resp.~$Q_2$, $Q_3$) intersects $N$ only at the point
$P_1=\phi(-1)$ (resp.~$P_2=\phi(-\omega)$, $P_3=\phi(-\omega^2)$).

Conversely, if $Q$ is an irreducible conic with the property that
$Q$ intersects $N$ only at a point $P\in N\setminus \{O\}$,
then $Q$ is one of the above three conics.

Note that the tangent line to $Q_i$ at $P_i$ passes through
$O_i$ for each $i$ and that $\varphi_i(Q_i)=Q_i$, 
$\varphi_i(Q_j)=Q_k$ for distinct
$i$, $j$, $k$ among $\{1,2,3\}$.
\end{theorem}
\begin{proof}
Let $Q$ be a conic defined by the general equation:
\[
ax^2+by^2+cz^2+dxy+exz+fyz=0.
\]
Suppose that $Q$ intersects $N$ only at a point
$P=\phi(\alpha)\in N\setminus\{O\}$,
where $\alpha\in\SC^{\ast}$.
Then we have
\[
at^2+bt^4+c(t^3-1)^2-dt^3+et(t^3-1)-ft^2(t^3-1)=0.
\]
It follows that
\[
ct^6-ft^5+(b+e)t^4-(2c+d)t^3+(a+f)t^2-et+c=0.
\]
Since $Q$ does not pass through $O$,
we infer that $c\neq 0$.
So we may assume that $c=1$.
Thus, we have
\[
t^6-ft^5+(b+e)t^4-(2+d)t^3+(a+f)t^2-et+1=0.
\]
By our hypothesis, this equation must have only one multiple root $\alpha$
of order six.
We see that
$\alpha^6=1$, $f=6\alpha$, $b+e=15\alpha^2$, $2+d=20\alpha^3$,
$a+f=15\alpha^4$, $e=6\alpha^5$.
In particular, $\alpha$ is a 6-th root of unity.
We then obtain the equations of the conics $Q_1$, $Q_2$, $Q_3$
for $\alpha=-1$, $-\omega$, $-\omega^2$, respectively.
For the cases in which $\alpha=1$, $\omega$, $\omega^2$,
the conic $Q$ is reduced to a double tangent line
at the flex $O_1$, $O_2$, $O_3$, respectively.
\end{proof}

\ \\
\textsc{%
{\small
Department of Mathematics,
Faculty of Science,
Saitama University,\\
Shimo-Okubo 255,
Urawa Saitama 338--8570,
Japan.}}\\
{\small
\textit{E-mail address}:  \texttt{ktono@rimath.saitama-u.ac.jp}
}
\end{document}